\documentclass[a4paper,10pt]{article}

\usepackage{amsmath,amsthm,amssymb}
\usepackage{tikz}
\usepackage[all]{xy}
\usepackage[pagebackref=true,colorlinks,linkcolor=magenta,citecolor=blue]{hyperref}

\newtheorem{theorem}{Theorem}[section]
\newtheorem{proposition}[theorem]{Proposition}
\newtheorem{lemma}[theorem]{Lemma}
\newtheorem{corollary}[theorem]{Corollary}
\newtheorem{claim}[]{Claim}
\theoremstyle{definition}
\newtheorem{definition}[theorem]{Definition}
\newtheorem{remark}[theorem]{Remark}

\newtheorem{example}[theorem]{Example}

\newcommand{\mfa}{\mathfrak{a}}
\newcommand{\mfm}{\mathfrak{m}}
\newcommand{\mfp}{\mathfrak{p}}
\newcommand{\mcF}{\mathcal{F}}
\newcommand{\mcJ}{\mathcal{J}}
\newcommand{\mcL}{\mathcal{L}}
\newcommand{\bfa}{\mathbf{a}}
\newcommand{\bfx}{\mathbf{x}}
\newcommand{\kk}{\Bbbk}
\newcommand{\CC}{\mathbb{C}}
\newcommand{\NN}{\mathbb{N}}
\newcommand{\PP}{\mathbb{P}}
\newcommand{\ZZ}{\mathbb{Z}}

\newcommand{\Hom}{\operatorname{Hom}}
\newcommand{\Ext}{\operatorname{Ext}}
\newcommand{\height}{\operatorname{height}}
\newcommand{\supp}{\operatorname{supp}}
\newcommand{\spec}{\operatorname{spec}}
\renewcommand{\H}{H}
\newcommand{\tH}{\widetilde{H}}
\newcommand{\tC}{\widetilde{C}}
\newcommand{\lcm}{\operatorname{lcm}}

\author{Amin Nematbakhsh}
\title{Linear strands of edge ideals of multipartite uniform clutters}
\date{}

\begin{document}
\maketitle

\begin{abstract}
We construct the first linear strand of the minimal free resolutions of edge ideals of $d$-partite $d$-uniform clutters.
We show that the first linear strand of such ideals are supported on relative simplicial complexes.
In the case that the edge ideals of such clutters have linear resolutions, we give an explicit and surprisingly simple description of their minimal free resolutions, generalizing known resolutions for edge ideals of Ferrers graphs and hypergraphs and co-letterplace ideals.
As an application, we show that the Lyubeznik numbers that appear on the last column of the Lyubeznik table of the cover ideal of such clutters are Betti numbers of certain simplicial complexes.
Furthermore, we restate a characterization for edge ideals of $d$-partite $d$-uniform clutters which have linear resolutions based on the recent characterization of arithmetically Cohen-Macaulay sets of points in multiprojective spaces.
\end{abstract}

\section{Introduction}
\label{sec-intro}

Classification of minimal free resolutions of monomial ideals is one of the central open problems in combinatorial commutative algebra.
There exists a variety of methods to compute free resolutions of monomial ideals (e.g. Taylor complex or Lyubeznik Complex) but construction of minimal free resolutions remains a challenging problem.
There are only a few classes of monomial ideals for which an explicit minimal free resolution is known. The most celebrated examples are generic and Borel ideals.  
A monomial ideal in a polynomial ring $\kk[x_1,\ldots,x_n]$ is $\ZZ^n$-graded and the $\ZZ^n$-graded Betti numbers can be computed via different algebraic and combinatorial methods.
This provides us with a description of the terms in a minimal $\ZZ^n$-graded free resolution. The construction of maps in a minimal free resolution is still an open problem in general.

Let $I$ be a squarefree monomial ideal in a polynomial ring $R$, such that its Alexander dual contains a regular sequence of length $d$ consisting of squarefree monomials, where $d$ is the minimum degree of the generators of $I$.
We call such a regular sequence an {\it admissible sequence}.
In Theorem \ref{thm-main}, we give an explicit description of the first linear strand of the minimal free resolution of $I$.
The prototypes of ideals having admissible sequences are the edge ideals of $d$-partite $d$-uniform clutters.
In this article, we construct the first linear strand of such ideals based on the work of K.~Yanagawa on squarefree modules \cite{Yanagawa-01,Yanagawa-02}.
Yanagawa uses the graded structure of the $\Ext$ module $\Ext^i(R/I^A,\omega_R)$ to construct the the $i$-th linear strand of the minimal free resolution of $I$, where $I^A$ is the Alexander dual of $I$ and $\omega_R$ is the canonical module of $R$.
This method is essentially used in \cite{DAli-Floystad-Nematbakhsh-02} to construct the minimal free resolution of co-letterplace ideals.
The main challenge of this technique is to find a description for the $\Ext$ modules.
Here, we provide an easy description of $\Ext^d(R/I^A,\omega_R)$ using linkage.
Then by Yanagawa's construction, we get the first linear strand of the minimal free resolution of $I$.
In particular, if the ideal $I$ has a linear resolution then we get an explicit description of its minimal free resolution.

\subsection*{Linear resolutions}
Many classes of ideals for which their minimal linear free resolutions are constructed in the literature are either edge ideals of $d$-partite $d$-uniform clutters or are specializations (quotients by regular sequences of linear forms) of such ideals.
Up to our knowledge these classes include: Edge ideals of Ferrers graphs and hypergraphs \cite{Corso-Nagel-01,Nagel-Reiner-01}, strongly stable and squarefree strongly stable hypergraphs and their associated $d$-partite $d$-uniform hypergraphs \cite{Nagel-Reiner-01}, co-letterplace ideals \cite{Ene-Herzog-Mohammadi-01,DAli-Floystad-Nematbakhsh-02}, strongly stable ideals generated in a single degree \cite{DAli-Floystad-Nematbakhsh-02}, uniform face ideals \cite{CookII-01} and edge ideals of
cointerval d-hypergraphs \cite{Dochtermann-Engstrom-01}.
See also examples \ref{exa-01} to \ref{exa-03}.

\medskip
If we replace the matrices of the linear strand with their so called monomial matrices (see \cite[Chapter 4]{Miller-Sturmfels-01} for its definition) then we get the simplicial chain complex of a relative simplicial complex.
In analogy with the theory of cellular resolutions, we say that the first linear strand is supported on a relative simplicial complex, Theorem \ref{thm-main2}.

\subsection*{Complexes supported on relative simplicial complexes}

In the theory of cellular resolutions, one starts with labeling the vertices of a simplicial complex (or more generally a polyhedral cell complex) $X$ with monomials in a polynomial ring $R$.
The faces of $X$ are also labeled with the least common multiples of monomial labels of its vertices.
By using monomial matrices, the simplicial chain complex of such a labeled complex gives us a chain complex of free $R$-modules.
The theory studies conditions for which this complex of free $R$-modules supported on $X$ provides a free resolution of the ideal generated by monomial labels of vertices of $X$.
The construction of cellular resolutions can be broadened to include resolutions supported on relative simplicial complexes, \cite[Definition 5.2]{Miller-02}.
More precisely, 
one starts with a labeled simplicial complex (or more generally with a polyhedral cell complex) $X$ with vertex set $V$.
Let $R = \kk[V]$, i.e. a polynomial ring with vertices of $X$ as indeterminates.
Then we remove a subcomplex $Y$ of $X$ and seek conditions for which the complex of free $R$-modules supported on the relative simplicial complex $(X,Y)$ is a (minimal) free resolution (or the first linear strand of a minimal free resolution) of the ideal generated by monomials associated to faces of minimal dimension in $(X,Y)$.
We show that for the class of edge ideals of $d$-partite $d$-uniform clutters, this is the case and the first linear strand of their minimal free resolutions are supported on relative simplicial complexes.
Furthermore, we show that the vertices of these relative simplicial complexes are labeled with single variables, making them useful tools to give descriptions for the Betti numbers. 

\medskip
The Betti numbers of a monomial ideal with a cellular resolution can be computed using the cellular structure.
If a monomial ideal has a minimal free resolution supported on a relative simplicial complex $(X,Y)$ (in the sense above)
in which the vertices are labeled with distinct variables, 
then the sequence of (total) Betti numbers $\beta_0(I),\beta_1(I), \ldots$ is nothing but the $f$-vector of the relative simplicial complex $(X,Y)$.
Existence of linear strands supported on relative simplicial complexes has direct applications in computation of Lyubeznik numbers, especially when the subcomplex $Y$ is contractible or a homology sphere.

\subsection*{Applications to the computation of Lyubeznik numbers}
Lyubeznik numbers are homological invariants of a ring introduced by Gennady Lyubeznik in \cite{Lyubeznik-01}. They are defined as Bass numbers of certain local cohomology modules, see Section \ref{sec-Lyubeznik}. Similar to Betti numbers these invariants can be collected in a table called the {\it Lyubeznik table}.
For a squarefree monomial ideal $I$, the Lyubeznik numbers of the ring $R/I$ are closely related to the homology of linear strands of a minimal free resolution of $I^A$, where $I^A$ is the Alexander dual of $I$.
These invariants are usually very difficult to compute. Except for Cohen-Macaulay ideals for which their Lyubeznik tables are trivial, there are only a very few classes of ideals for which their Lyubeznik tables are known.
Finding families of examples for which their Lyubeznik tables can be described is one of the main questions in the theory of Lyubeznik numbers proposed by J. {\`A}lvarez-Montaner in the expository article \cite{AlvarezMontaner-01}.
Another intriguing question is what kind of topological information is provided by Lyubeznik numbers.
For example, if $X$ is a scheme of finite type over $\CC$ with an isolated singularity at the point $x\in X$ and $R$ is the local ring $\mathcal{O}_{X,x}$ with $\dim R \geq 2$, then every Lyubeznik number of $R$ equals the dimension of a certain singular cohomology group of $X$ with coefficients in $\CC$ and support on $x$ \cite{GarciaLopez-Sabbah-01}.
Later similar connections between Lyubeznik numbers and {\'e}tale cohomology in positive characteristic was given in \cite{Blickle-Bonda-01} and \cite{Blickle-01}. 
We show that for the cover ideals of $d$-partite $d$-uniform clutters the Lyubeznik numbers sitting on the last column of the Lyubeznik table are Betti numbers of certain simplicial complexes.
This partially answers a question in \cite[Section 7]{AlvarezMontaner-Vahidi-01} that whether there are cellular structures on the linear strands of a free resolution so one can give topological descriptions of the Lyubeznik numbers.
Moreover, when the edge ideals of such clutters have linear resolutions then we get a complete picture of the Lyubeznik table of their Alexander dual ideals.

\medskip
Suppose that $I$ is the edge ideal of a $d$-partite $d$-uniform clutter.
If $I$ has a linear resolution then obviously its first linear strand is its resolution and in Theorem \ref{thm-main} we get a description of its minimal free resolution. Furthermore, Theorem \ref{thm-Lyu} describes the full Lyubeznik table of its Alexander dual (since the only nontrivial column is the last one).
Therefore, it would be interesting to have a characterization for edge ideals of $d$-partite $d$-uniform clutters which have linear resolutions.
This characterization was achieved by G.~Favacchio, E.~Guardo and J.~Migliore while establishing a characterization of arithmetically Cohen-Macaulay sets of points in the multiprojective space $(\PP^{1})^{\times d}$. 

\subsection*{Configuration of points in the multiprojective space}
In \cite{Favacchio-Guardo-Migliore-01}, G.~Favacchio et al. introduced a correspondence between finite reduced sets of points in the multiprojective space $(\PP^1)^{\times d}$ and $d$-partite $d$-uniform clutters. 
They show that for any finite reduced set of points $X$ in $(\PP^1)^{\times d}$ one can associate a $d$-partite $d$-uniform clutter $C$. The associated $d$-partite $d$-uniform clutter has a linear resolution if and only if $X$ is arithmetically Cohen-Macaulay. For any $d$-partite $d$-uniform clutter $C$, there exists many sets of points in $(\PP^1)^{\times d}$ for which the associated clutter is $C$.
Therefore, their characterization of arithmetically Cohen-Macaulay sets of points in $(\PP^1)^{\times d}$ provides a characterization for $d$-partite $d$-uniform clutters with linear resolutions.
They show that a finite set of points is arithmetically Cohen-Macaulay if and only if it satisfies the so called $(\star_d)$-property. The $(\star_d)$-property provides a characterization for clutters for which their edge ideals have linear resolutions. Since this characterization is given in a different language and also since the result is not known to many people working in the field of combinatorial commutative algebra, we decided to include this characterization in an appendix section.
It follows from the Fr\"oberg's theorem that the edge ideal of a bipartite graph has a linear resolution if and only if it does not have any induced subgraph consisting of two disjoint edges.
We rephrase the characterization of $d$-partite $d$-uniform clutters for which their edge ideals have linear resolutions in Proposition \ref{pro-linearchar}. It is formulated in a setting that resembles the similar result on bipartite graphs.

\subsection*{Structure of the paper}
In Section \ref{sec-pre}, we recall the main terminology and tools that will play a central role in the statements and proofs of the subsequent sections.
In Section \ref{sec-main}, we introduce the notion of an admissible regular sequence. The linear strand of a squarefree monomial ideal which admits an admissible regular sequence is given in Theorem \ref{thm-main}.
In particular, Theorem \ref{thm-main} gives a description for the first linear strand of the minimal free resolutions of the class of edge ideals of $d$-partite $d$-uniform clutters as well as the class of flag ideals of posets.
We show in Theorem \ref{thm-main2} that the linear strand is supported on a relative simplicial complex.
Section \ref{sec-Lyubeznik} contains the applications of Theorem \ref{thm-main} and Theorem \ref{thm-main2} in computation of Lyubeznik numbers.
In Theorem \ref{thm-Lyu} we compute the last column of the Lyubeznik table of the cover ideal of a $d$-partite $d$-uniform clutter in terms of the Betti numbers of certain simplicial complexes (or equivalently the Betti numbers of certain squarefree monomial ideals by Hochster's formula).
The characterization of edge ideals of $d$-partite $d$-uniform clutters that have linear resolutions is given in the Appendix section. 

\subsection*{Acknowledgment}
The author would like to thank the referee for the valuable comments which helped to
improve the manuscript.
This research was supported by a grant from the Institute for Research in Fundamental Sciences (IPM), Tehran, Iran.

\section{Preliminaries}
\label{sec-pre}
In this section, we provide the terminology and preliminary materials that we need through the article.

\subsection{Monomial ideals and minimal free resolutions}

Let $I$ be a monomial ideal in a polynomial ring $R = \kk[x_1,\ldots,x_n]$.
We denote the unique set of minimal generators of $I$ by $G(I)$. We also denote the set of minimal generators of degree $d$ by $G_d(I)$.

For a subset $W=\{x_{i_1},\ldots,x_{i_k}\}$ of $\{x_1,\ldots,x_n\}$, we denote the product of elements of $W$ by $m_W$, i.e. $m_W = x_{i_1}\cdots x_{i_k}$.
The {\it support} of a monomial $m$, denoted by $\supp(m)$, is the set of variables that divide $m$.

Let $I$ be a squarefree monomial ideal.
The {\it Alexander dual} of $I$ is the ideal generated by all monomials $m$ such that $m$ has a nontrivial common divisor with all elements of $I$. We denote the Alexander dual of $I$ by $I^A$.

A minimal free resolution of a squarefree monomial ideal $I$ in $R=\kk[x_1,\ldots,x_n]$, is an exact sequence of free $R$-modules,
\[
\xymatrix{
F_\bullet: 0 \ar[r] & F_n \ar[r]^{d_n} & F_{n-1} \ar[r]^{d_{n-1}} &  \cdots \ar[r] & F_1 \ar[r]^{d_1} & F_0 \ar[r] & I \ar[r] & 0
}
\]
where $d_i(F_i) \subseteq (x_1,\ldots,x_n)F_{i-1}$, for $i=1,\ldots,n$.
Since a monomial ideal is $\ZZ^n$-graded, each $F_i$ is a direct sum of $R$ with degree shifts,
\[
F_i = \bigoplus_{\bfa \in \ZZ^n} R(-\bfa)^{\beta_{i,\bfa}}.
\]
The numbers $\beta_{i,\bfa}$ are invariants of the ideal $I$ and are called the {\it multigraded (or fine) Betti numbers} of $I$. The {\it graded (or coarse) Betti numbers} of $I$ are defined as the sum over multidegrees of the same degree,
\[
\beta_{i,j} = \sum_{\substack{\bfa \in \ZZ^n\\|\bfa| = j}} \beta_{i,\bfa}.
\]
where for $\bfa = (a_1,\ldots,a_n) \in \ZZ^n$,  $|\bfa| = a_1+\cdots+a_n$. 
The {$r$-th linear strand} of $I$ is the subcomplex of $F_\bullet$,
\[
\xymatrix{
F_\bullet^{<r>} : 0 \ar[r] & F^{<r>}_{n-r} \ar[r] & F^{<r>}_{n-r-1} \ar[r] & \cdots \ar[r] & F_1^{<r>} \ar[r] & F_0^{<r>} \ar[r]& 0 
}
\]
where, 
\[
F_i^{<r>}  = \bigoplus_{\substack{\bfa \in \ZZ^n\\|\bfa| = i + r}} R(-\bfa)^{\beta_{i,\bfa}}.
\]
If $r$ is the smallest integer for which the $r$-th linear strand is nontrivial then the $r$-th linear strand of $I$ is usually called the {\it first linear strand} of $I$. 

\subsection{Hypergraphs and clutters}
A {\it hypergraph} $C$ is a pair $(V(C),E(C))$, where $V(C)$ is a set of elements called {\it vertices}, and $E(C)$ is a set of nonempty subsets of $V(C)$ called {\it edges}. If no edge of the hypergraph $C$ is a subset of another edge of $C$, then $C$ is called a {\it clutter}. A clutter $C$ is called
\begin{enumerate}
\item {\it $d$-uniform} if all edges of $C$ have cardinality $d$,
\item {\it $d$-partite} if there is a partition of the vertex set $V(C)$, i.e. $V(C)=V_1 \cup \cdots \cup V_d$ such that for every edge $e$ of $G$, each vertex of $e$ belongs to exactly one $V_i$, $i=1,\ldots,d$.
The partition $V_1\cup \cdots\cup V_d$ of $V(C)$ is called a {$d$-partition} of the clutter $C$.
\end{enumerate}
In particular, a $2$-uniform clutter is a simple graph and a $2$-partite $2$-uniform clutter is a bipartite graph.

A subset $D\subseteq V(C)$ is called a {\it vertex cover} if it has a nonempty intersection with any edge of $C$. A {\it minimal vertex cover} is a vertex cover for which any proper subset is not a vertex cover.
A subset $A\subseteq V(C)$ is called an {\it independent set} if it does not contain any edge of $C$. A {\it maximal independent set} is an independent set which is not a proper subset of another independent set. A subset $D$ is a vertex cover if and only if its complement in $V(C)$ is an independent set.

There is a natural correspondence between clutters on a vertex set $V$ and square-free monomial ideals in the polynomial ring $\kk[V]$ with elements of $V$ as indeterminates.
For any clutter $C$, the corresponding ideal is the ideal $I(C) = (m_e~|~e\in E(C))$ called the {\it edge ideal} of $C$.
To any clutter $C$ one can assign yet another squarefree monomial ideal called the {\it cover ideal} of $C$. By definition it is the monomial ideal generated by all monomial $m_D$, where $D$ ranges over minimal vertex covers of $C$.
It is easy to show that the cover ideal of $C$ is equal to the Alexander dual of the edge ideal of $C$.

\subsection{Squarefree modules and Yanagawa construction}

Let $R = \kk[x_1,\ldots,x_n]$ be a polynomial ring with its natural $\NN^n$-grading.
Let $\epsilon_1,\ldots,\epsilon_n$ be the unit coordinate vectors of $\NN^n$.
For each multidegree $\bfa=(a_1,\ldots,a_n) \in \NN^n$, we may also call the monomial $\bfx^\bfa = x_1^{a_1}\cdots x_{n}^{a_n}$ a multidegree.
A finitely generated module $M$ over $R$ is called a {\it squarefree module} (or {\it positively $\mathbf{1}$-determined module}) if for each multidegree $\bfa\in \NN^n$ the multiplication map $M_\bfa \stackrel{x_i}{\longrightarrow} M_{\bfa+\epsilon_i}$ is an isomorphism whenever $a_i\geq 1$.
Squarefree modules were introduced by K.~Yanagawa in \cite{Yanagawa-01} and later generalized to positively $\bfa$-determined modules ($\bfa \in \NN^n$ and $a_i \geq 1$ for all $i$) by E.~Miller in \cite{Miller-01}.

The notion of Alexander duality for squarefree monomial ideals extends to squarefree modules.
For a squarefree module $M$, define another squarefree module $M^\ast$ as below.
\begin{enumerate}
\item For $\bfa \in \{0,1\}^n$, let $M^\ast_\bfa = \Hom_\kk(M_{\mathbf{1}-\bfa},\kk)$;
\item When $a_i = 0$, define the multiplication  map $M^\ast_\bfa \stackrel{x_i}{\longrightarrow} M^\ast_{\bfa+\epsilon_i}$ to be the dual of the multiplication map $M_{\mathbf{1}-\bfa-\epsilon_i} \stackrel{x_i}{\longrightarrow} M_{\mathbf{1}-\bfa}$ of $M$.
\item The rest of multigraded parts and the multiplication maps among them is defined by obvious extensions.
\end{enumerate}
The squarefree module $M^\ast$ defined above is called the {\it Alexander dual} of $M$.
If $I$ is a squarefree monomial ideal then $I^\ast = R/I^A$.

Any multidegree $\bfa\in \{0,1\}^n$ has the form $\bfa =\sum_{j\in J} \epsilon_j$ for some subset $J\subseteq [n]$.
Define its complement $\bfa^c$ to be the multidegree $\bfa^c=\sum_{j\in J^c} \epsilon_j$.
For a squarefree $R$-module $M$, K.~Yanagawa \cite{Yanagawa-01} introduced a complex of free $R$-modules denoted by $\mcL_\bullet(M)$ and defined with terms and differentials as follows (also see \cite{DAli-Floystad-Nematbakhsh-02}).
\begin{enumerate}
\item For $i\in \ZZ$,
\[\mcL_i(M) = \bigoplus_{\substack{\bfa\in \{0,1\}^n,\\ |\bfa^c| = i}} M_\bfa^\circ \otimes_\kk R\]
where $M_\bfa^\circ$ is the same as $M_\bfa$ but assumed to have multidegree $\bfa^c$.
\item For $i\in \ZZ$, the differential $d_i : \mcL_i(M) \to \mcL_{i+1}(M)$ is given by
\[
m^\circ \otimes s \mapsto \sum_{j\in \bfa^c} (-1)^{\alpha(j,\bfa)} (x_jm)^\circ \otimes x_js
\]
where $\alpha(j,\bfa) = \{i| i<j\}$.
\end{enumerate}

In the same article, it is shown that if $M$ is squarefree then all of its syzygy modules and the $\Ext$ modules $\Ext^i_R(M, \omega_R)$ are squarefree, as well.
Let $I$ be a square-free monomial ideal of $R$ and let $d$ be the smallest degree of generators of $I$.
For the particular case of $M = \Ext^i_R(R/I^A, R(-\mathbf{1}))$ the chain complex $\mcL_\bullet(M)[i]$ is isomorphic to the $i$-th linear strand of the resolution of the ideal $I$.
If in addition, the ideal $I$ has a $d$-linear resolution then by a result of J.~A.~Eagon and V.~Reiner \cite{Eagon-Reiner-01} $R/I^A$ is a Cohen--Macaulay ring.
In this case, the only non-vanishing Ext module is the top one which is the canonical module and the chain complex $\mcL_\bullet(\omega_{R/I^A})$ is the
linear resolution of the ideal $I$.

\subsection{Relative simplicial (co)homology}

Let $X$ be a simplicial complex and let $Y$ be a subcomplex of $X$.
The {\it relative simplicial complex} $(X,Y)$ is the collection of all faces of $X$ which are not in $Y$, i.e. $X\backslash Y$.
An orientation of $X$ induces an orientation of $Y$ and one can construct the {\it (reduced) relative simplicial chain complex} of $(X,Y)$ (or {\it (reduced) relative simplicial chain complex} of $X$ with respect to $Y$) in the same manner as the construction of the reduced chain complex of a simplicial complex.
We denote the (reduced) simplicial chain complex of $X$ by $\widetilde{C}_\bullet(X)$ and its (reduced) relative chain complex with respect to $Y$ by $\widetilde{C}_\bullet(X,Y)$.
The complex $\widetilde{C}_\bullet(X,Y)$ is isomorphic to the cokernel of the natural inclusion map
$\widetilde{C}_\bullet(Y) \rightarrow \widetilde{C}_\bullet(X)$.
In analogy with the reduced homology of a simplicial complex, the $i$-th homology of the complex $\widetilde{C}_\bullet(X,Y) \otimes \kk$ is called the $i$-th {\it (reduced) relative homology of the pair $(X,Y)$ with coefficient in $\kk$} and is denoted by $\tH_i((X,Y);\kk)$.
Since we always assume that $Y$ is nonempty, $\tH_i((X,Y);\kk) = H_i((X,Y);\kk)$ for all $i$.
The (reduced) relative cohomology is defined similar to reduced cohomology of a simplicial complex.
For detail we refer to standard books on algebraic topology like \cite{Munkres-01} or \cite{Hatcher-01}.

For any subset $W$ of the vertex set of $X$, the {\it restriction} of $X$ to $W$ is the simplicial complex
$X_{|W} = \{F\in X ~|~ F\subseteq W\}$. Similarly, the {\it restriction} of $(X,Y)$ to $W$ denoted by $(X,Y)_{|W}$ is the relative simplicial complex consisting of all faces of $(X,Y)$ which are contained in $W$, i.e. $(X,Y)_{|W} = (X_{|W},Y_{|W\cap \{\text{vertex set of } Y\}})$.
We call a simplicial complex {\it acyclic} if it is either empty or has zero reduced homology.

\section{Linear strands of edge ideals of $d$-partite $d$-uniform clutters}
\label{sec-main}
\begin{definition}
Let $V_1,\ldots,V_d$ be finite disjoint sets.
By the {\it complete $d$-partite clutter on the vertex set $V=V_1\cup\cdots\cup V_d$}, we mean the the $d$-partite $d$-uniform clutter on vertex set $V$ defined by the edge set
\[
\{\{v_1,\ldots,v_d\} ~|~ v_i \in V_i, i=1,\ldots,d\}
\]
and we denote it by $C({V_1,\ldots,V_d})$ or $C({n_1,\ldots,n_d})$, where $n_i = |V_i|$ for $i=1,\ldots,d$.
\end{definition}

Let $C$ be a clutter on the vertex set $V$ and let $R=\kk[V]$ be a polynomial ring with elements of $V$ as indeterminates.
Let $W$ be a subset of $V$.
The {\it restriction of $C$ to $W$} (or simply the {\it induced clutter on $W$} when $C$ is fixed) denoted by $C_{|W}$ is the clutter with vertex set $W$ and edge set $E(C_{|W})$ consisting of the edges $e \in E(C)$, such that $e\subseteq W$.
The {\it projection of $C$ on $W$} denoted by $C^W$ is the clutter with vertex set $W$ and edge set
\[
E(C^W) = \{e \cap W ~|~ e\in E(C), e\cap W \neq \emptyset\}.
\]
Suppose $C$ is a $d$-partite $d$-uniform clutter with $d$-partition $V_1\cup \cdots\cup V_d$.
If for some $J\subseteq [d]$, $W = \cup_{i\in J}V_i$, then we call the clutter $C^W$ a {\it ranked projection of rank $|J|$}.
The ranked projections are used to characterize edge ideals of $d$-partite $d$-uniform clutters with linear resolutions, see Proposition \ref{pro-linearchar}.

Let $C$ be a $d$-partite $d$-uniform clutter on the $d$-partition $V=V_1\cup\cdots\cup V_d$. Consider the complete $d$-partite clutter $C({V_1,\ldots,V_d})$. The complement of $E(C)$ in the edge set of $C({V_1,\ldots,V_d})$, defines a $d$-partite $d$-uniform clutter, which we denote by $C^c$ and call it the {\it $d$-partite complement} of $C$.

Before giving the main theorem, we recall a lemma in homological algebra which is a useful tool in the computation of Ext modules.

\begin{lemma}[Lemma 1.2.4 \cite{Bruns-Herzog-01}]
\label{thm-HomExt}
Let $M$ and $N$ be $R$-modules and $a_1,\ldots,a_n$ be a regular sequence on $N$. If $(a_1,\ldots,a_n)M=0$ then
\[
\Ext^n_R(M,N) = \Hom_R(M,N/(a_1,\ldots, a_n)N).
\]
\end{lemma}

\begin{definition}
\label{def-admis}
Let $I$ be a squarefree monomial ideal.
Let $d(I)$ be the smallest degree of generators of $I$.
We call a regular sequence of monomials $a_1,\ldots,a_d$, an {\it admissible regular sequence} if
\begin{itemize}
\item $d = d(I)$;
\item the ideal $\mfa = <a_1,\ldots,a_d>$ is contained in $I^A$.
\end{itemize}
Note that a sequence of monomials $a_1,\ldots,a_d$ is a regular sequence if and only if the monomials are pairwise coprime.
Obviously, $\height(\mfa) = \mu(\mfa)$, where $\mu(\mfa)$ denotes the minimum number of generators of $\mfa$. Hence $\mfa$ is a complete intersection.
\end{definition}

The following examples present two large classes of squarefree monomial ideals that have admissible regular sequences.

\begin{example}
Let $C$ be a $d$-partite $d$-uniform clutter with a $d$-partition $V=V_1\cup\cdots\cup V_d$.
The sequence $m_{V_1},\ldots,m_{V_d}$ form an admissible regular sequence for the ideal $I(C)$.
\end{example}

\begin{example}
Let $P$ be a finite poset. Let $\mcF(P)$ be the facet ideal of the order complex of $P$. In other words,
\[
\mcF(P) = < m_W~|~W \text{ is a maximal chain in }P >.
\]
The ideal $\mcF(P)$ is a squarefree monomial ideal in the polynomial ring $\kk[P]$ and it is called the {\it flag ideal of $P$} in \cite{Nematbakhsh-02}.
A rank function for $P$ is a function $r :P\to \NN$ such that for any two elements $p,q\in P$ for which $q$ covers $p$, $r(q) = r(p)+1$.
A poset is called {\it graded} if it admits a rank function.
Now suppose $P$ is graded and let $r_P$ be the unique rank function of $P$ that maps all of the minimal elements of $P$ to $1$.
For any $p\in P$, $r_P(p)$ is called the rank of $p$ and we denote the set of elements of $P$ of rank $i$ with $P_i$.
For flag ideals, the integer $d=d(\mcF(P))$ is the smallest length of maximal chains of $P$.
Obviously, $m_{P_1},\ldots,m_{P_d}$ form an admissible regular sequence for $\mcF(P)$.
The algebraic and homological properties of flag ideals are studied in \cite{Nematbakhsh-02}.
\end{example}

Let $I$ be a squarefree monomial ideal that has an admissible regular sequence $a_1,\ldots,a_d$.
Let $C$ be the clutter associated with $I$.
Let $\mfa = <a_1,\ldots,a_d>$ and
for $i=1,\ldots,d$, let $\mathfrak{p}_i$ be the the prime monomial ideal generated by elements of $\supp(a_i)$.
Suppose $J$ is the Alexander dual of the ideal $(\mfa:I^A)$.
In the following lemma we provide a description for the generators of $J$.
Recall that two proper ideals $I$ and $J$ of height $d$ in a commutative Noetherian ring are said to be {\it linked} if there exist a regular sequence $a_1,\ldots,a_d$ contained in $I\cap J$ such that $(a_1,\ldots,a_d):I = J$ and $(a_1,\ldots,a_d):J=I$.

\begin{lemma}
\label{lem-asli}
The ideal $J$ is the monomial ideal generated by the monomials in $G(\mfa^A) \backslash G_d(I)$.
In other words, $J$ is the edge ideal of the $d$-partite $d$-uniform clutter $C'$ defined on the $d$-partition $\bigcup_{i=1}^d \supp(a_i)$ and with the edge set 
\[E(C') = E(C(\supp(a_1),\ldots,\supp(a_d))) \backslash \{e\in E(C)~|~|e| = d\}.\]
\end{lemma}

\begin{proof}
By the definition of $J$,
\[
J^A= (\mfa : I^A) = \bigcap_{m\in G(I^A)} (\mfa:m).
\]
For any monomial $m\in G(I^A)$, we have that $m=m_D$ for a given minimal vertex cover $D$ of $C$.
Note that $(\mfa : m) = (\mfa: \gcd(m,m_W))$, where 
\[W =\bigcup_{i=1}^d \supp(a_i).\]
We show that the intersection above can be taken over all $m\in G(I(C_{|W}))$.
For a minimal vertex cover $D$ of $C$, $D\cap W$ is a vertex cover for $C_{|W}$, and we have 
$(\mfa:m_D) = (\mfa:m_{D\cap W})$, since $m_{D\cap W} = \gcd(m_D,m_W)$.
Conversely, suppose $D$ is a minimal vertex cover of $C_{|W}$.
Let $A$ be the complement of $D$ in $W$. The set $A$ is an independent set of $C$ and we extend it to a maximal independent set $\bar{A}$.
Denote its complement in $V$ by $\bar{D}$. Obviously, $\bar{D}$ is a minimal vertex cover of $C$ and $\bar{D} = D \cup V'$ where $V' \subseteq (V \backslash W)$.
Again $(\mfa:m_D) = (\mfa:m_{\bar{D}})$.
Therefore,
\[
(\mfa : I^A) = \bigcap_{m\in G(I^A)} (\mfa:m) = \bigcap_{m\in G(I(C_{|W}))} (\mfa:m) = (\mfa : I(C_{|W})).
\]
Note that $\mfa \subseteq I^A \subseteq I(C_{|W})^A$ and $C_{|W}$ is a $d$-partite $d$-uniform clutter.
The support of any minimal generator of $\mfa$ intersects any edge of $C_{|W}$.
Hence $C_{|W}$ is a $d$-partite $d$-uniform clutter and we can reduce to the case that $C$ is a $d$-partite $d$-uniform clutter.


\begin{claim} Let $C$ be a $d$-partite $d$-uniform clutter with $d$-partition $V_1\cup\cdots\cup V_d$ and let $\mfa = <m_{V_1},\ldots,m_{V_d}>$. We claim that $J=I(C^c)$.
Furthermore, $(\mfa:I^A) = J^A$ and $(\mfa:J^A)=I^A$, i.e. $I^A$ and $J^A$ are linked.
\end{claim}

\begin{proof}[Proof of claim]
Let $K = I(C^c)$. We show that $K^A = (\mfa:I^A)$ and $I^A = (\mfa:K^A)$. This will complete the proof.
Let $m$ be a monomial in $K^A$. Since $K^A$ is generated by minimal vertex covers of $I(C^c)$, $\supp(m)$ is a vertex cover of $C^c$.
If $m\notin (\mfa:I^A)$, then there exists a monomial $m'\in I^A$, such that $mm'\notin \mfa$.
This implies that for each $i$, $1\leq i\leq d$, $\exists v_i\in V_i$ such that $v_i \nmid mm'$.
The union of these vertices gives an edge $e$ of the complete clutter $C(V_1,\ldots,V_d)$.
By construction, $e\cap \supp(m) = \emptyset = e\cap \supp(m')$, which is a contradiction, since $e$ either belongs to $C$ or belongs to $C^c$.

Conversely, let $m\in (\mfa:I^A)$ be a monomial.
If $m\notin K^A$, then there exists an edge in $C^c$, such that $\supp(m)\cap e = \emptyset$.
Let $D$ be the complement of $e$ in $V(C)$. Note that $\supp(m) \subseteq D$. It is easy to show that $D$ is a vertex cover of $C$.
Therefore, we have $m_D \in I^A$. This implies that $mm_D \in \mfa$, which is a contradiction, since $\supp(mm_D) \subseteq D$.
The proof of the other assertion is similar.
\end{proof}

It is easy to show that $I(C_{|W}^c)$ is exactly the ideal generated by $G(\mfa^A)\backslash G_d(I)$.

\end{proof}

Before we mention the main theorem we fix some notation.
Let $I$ be a squarefree monomial ideal with an admissible regular sequence $a_1,\ldots,a_d$.
Let $\mfa = <a_1,\ldots,a_d>$ and $W= \bigcup_{i=1}^d \supp(a_i)$.
Let $C$ be the clutter associated with $I$. Recall that $C_{|W}$ denotes the induced clutter on $W$. We see in the proof of Lemma \ref{lem-asli} that $C_{|W}$ is a $d$-partite $d$-uniform clutter with $d$-partition $\bigcup_{i=1}^d \supp(a_i)$.
We assume that the elements of $V(C)$ are totally ordered by a fixed order relation $\prec$.
In the following theorem, by $C^c_{|W}$ we mean $(C_{|W})^c$.

\begin{theorem}
\label{thm-main}
Let $I$ be a squarefree monomial ideal with an admissible regular sequence $a_1,\ldots,a_d$.
Let $W$ and $C$ be as above.
The first linear strand of a minimal free resolution of $I$ is isomorphic to the complex $\mcL_\bullet(I(C_{|W}^c)^A/\mfa)[d]$ with terms
\begin{equation}
\label{eq-main1}
F_i = \bigoplus_{\substack{D\text{ is a vertex cover of }C_{|W}^c,\\|D| = n-i-d,\\m_D\notin \mfa}} S(-D)
\end{equation}
and differentials,
\[
e_D \mapsto \sum_{\substack{v \in D^c\\m_{D \cup \{v\}} \notin \mfa}} (-1)^{\alpha(v,D)} v e_{D\cup \{v\} }
\]
where $\alpha(v,D) = \{w\in D| w \prec v\}$.
\end{theorem}

\begin{proof}
It follows from Lemma \ref{lem-asli} and its proof that the first linear strand of $I$ is the same as the first linear strand of the edge ideal $I(C_{|W})$.
Therefore, for the rest of the proof we assume that $I$ is the edge ideal of a $d$-partite $d$-uniform clutter $C$ with the $d$-partition $\cup_{i=1}^d \supp(a_i)$.

Let $J = I(C)^A$. By \cite[Corollary 4.2]{Yanagawa-01} the complex $\mcL(\Ext^d_R(R/J,R)[d]$ is isomorphic to the first linear strand of the minimal free resolution of $I(C)$.
For an edge $e\in E(C)$, let $\mfp_e = (e)$ be the monomial prime ideal generated by elements of $e$. By the definition of Alexander duality, we have $J = \bigcap_{e\in E(C)} \mfp_e$.
Let $K = I(C^c)^A$.
Recall that $K$ is generated by minimal vertex covers of $I(C^c)$.

Now we have
\[
K/\mfa \cong (\mfa:J)/\mfa \cong \Hom_R(R/J,R/\mfa) \cong \Ext_R^d(R/J,R(-\mathbf{1})).
\]
The first isomorphism follows from Lemma \ref{lem-asli} and the last one follows from Lemma \ref{thm-HomExt}.
Note that these are degree preserving isomorphisms of squarefree modules.
The first index $i$, for which $\mcL_i(K/\mfa)$ is nonzero is $i=d$.
Let $D$ be a subset of $V(C)$.
If $i<d$ and $|D^c| = i$, then $|D| > n-d$. This implies that there exist some $j$ such that $V_j \subseteq D$. Hence $m_D \in \mfa$.
For each $i\geq d$, the homogeneous monomials of degree $n-i$ in $K/\mfa$ form a basis of $\mcL_i(K/\mfa)$.
These monomials are precisely the vertex covers of $C^c$ of cardinality $n-i$, that does not contain any of the partition sets $V_j$, $j=1,\ldots,d$. Now it follows from the definition of the complex $\mcL_\bullet(K/\mfa)$ that the terms and differentials are exactly as given above.
\end{proof}

Let $X$ be a polyhedral cell complex with vertex set $V$.
Suppose that vertices of $X$ are labeled by monomials in the polynomial ring $R=\kk[V]$ with vertices of $X$ as indeterminates.
We also assume that each face of $X$ is labeled by the the least common multiple of the monomials associated to vertices in $F$.
The complex $X$ with such a labeling is called a {\it labeled cell complex}. 
Recall that a {\it monomial matrix} is a matrix $(\lambda_{p,q})$ of scalars in $\kk$ whose columns are labeled by monomials $m_p$ and whose rows are labeled by monomials $m_q$, such that $\lambda_{p,q} = 0$ unless $m_q | m_p$.
Let $\tC_\bullet = \tC_\bullet(X)\otimes_\ZZ \kk$ be the reduced chain complex of $X$ with coefficients in $\kk$.
Each map in $\tC_\bullet$ is given by a matrix of scalars for which each row and column corresponds to a face in $X$. One can think of these matrices as monomial matrices whose rows and columns are labeled by the  monomials associated to corresponding faces. The notion of a {\it monomial matrix} appears in the work of Ezra Miller and we refer to \cite[Chapter 4]{Miller-Sturmfels-01} for its definition.
Now if we replace each entry $\lambda_{p,q}$ in the monomial matrices by $\lambda \frac{m_p}{m_q}$ and we also change the source and target spaces to free $R$-modules accordingly, then we get a complex of free $R$-modules which is called the {\it cellular free complex supported on $X$} and is denoted by $\mcF_X$. We refer to \cite[Chapter 4]{Miller-Sturmfels-01} for more details and examples.
Let $I$ be the monomial ideal generated by monomials associated to vertices of $X$. If $\mcF_X$ is exact then we say that $I$ has a {\it cellular resolution supported on $X$}.
Cellular resolution were introduced by Bayer and B.~Sturmfels in \cite{Bayer-Sturmfels-01}.
It is known that any monomial ideal has a cellular resolution (the Taylor complex) but this resolution is not minimal in general. There are examples of monomial ideals that have no cellular minimal resolution \cite{Velasco-01}.
The construction of cellular resolutions can be broadened to relative structures. The notion of relative cellular resolution appears briefly at the end of Section 4.2 in \cite{Miller-01} which cites \cite{Miller-02}. In \cite{Miller-02}, E.~Miller studies the effect of Alexander duality on cellular resolutions and provides examples of resolutions supported on relative simplicial complexes.

In the sequel we show that the first linear strand of the resolution of an ideal which has an admissible regular sequence, is supported on a relative simplicial complex in the sense described below.
We start with a simplicial complex $X$ and we label each face $F$ with the monomial $m_F = \prod_{v\in F} v$.
The difference with the procedure above is that we label each vertex $v$ with the monomial $v$ instead of labeling it with an arbitrary monomial.
Then we remove a subcomplex $Y$ of $X$ to get the relative simplicial complex $(X,Y)$.
As before, we replace the matrices in the simplicial chain complex of $(X,Y)$ by monomial matrices to get a complex of free $R$-modules $\mcF_{(X,Y)}$.
Let $I$ be the ideal generated by monomial labels of faces of minimal dimension in $(X,Y)$.
If the complex $\mcF_{(X,Y)}$ is isomorphic to the first linear strand of the minimal free resolution of $I$ then we say that {\it the first linear strand is supported on a relative simplicial complex}.
More generally, one can seek examples where the construction above gives a free resolution of $I$ starting by a polyhedral cell complex.

Now let $I$ be a squarefree monomial ideal which has an admissible regular sequence.
In the terms of the resolution of $I$, instead of taking the sum over the vertex covers $D$ of $C_{|W}^c$ such that $m_D\notin \mfa$, one can take the sum over their complements, i.e. over the independent sets $A$ such that $\forall i$, $\gcd(A,a_i) \neq 1$.
Therefore, we can equivalently, define the first linear strand of $I$ by the terms
\begin{equation}
\label{eq-02}
F_i = \bigoplus_{\substack{A \text{ is an independent set of }C^c_{|W}\\|A| = i+d\\ \forall i, \gcd(a_i,m_A)\neq 1}} S(-A)
\end{equation}
and differentials
\[
e_A \mapsto \sum_{\substack{v\in A\\ \forall i, \gcd(a_i,m_{A\backslash \{v\}})\neq 1}} v e_{A\backslash v}.
\]
The set of independent sets of $C_{|W}^c$ is the Stanley-Reisner complex of $I(C_{|W}^c)$ by definition which we denote by $X(I)$.
Furthermore, the set of independent sets that does not satisfy the condition that for all $i=1,\ldots,d$, $\gcd(a_i,m_A) \neq 1$ forms a subcomplex of $X(I)$.
We denote this subcomplex by $Y(I)$.
If we replace the matrices in the resolution of $I$ by their monomial matrices (see \cite[Chapter 4]{Miller-Sturmfels-01} for its definition) then we get a chain complex of $\kk$-vector spaces which is exactly the chain complex 
\[
\tC_\bullet(X(I),Y(I))[d-1] \otimes_\ZZ \kk
\]
where $\tC_\bullet(X(I),Y(I))$ is the (reduced) relative simplicial chain complex of the relative simplicial complex $(X(I),Y(I))$.

We do not know if the first linear strand of $I$ is supported on a simplicial complex (in the sense of  \cite[Chapter 4]{Miller-Sturmfels-01}) but at least it is supported on a relative simplicial space.
We summarize the discussions above in the following theorem.

\begin{theorem}
\label{thm-main2}
Let $I$ be a squarefree monomial ideal that has an admissible regular sequence.
The first linear strand of $I$ is supported on the relative simplicial complex $(X(I),Y(I))$.
\end{theorem}

\begin{remark}
Let $C$ be a $d$-partite $d$-uniform clutter with vertex set $V$.
It follows from equation \ref{eq-02} that a multidegree $A \subseteq V$ appears in first linear strand of $I(C)$ if and only if $C_{|A}$ is a complete $d$-partite $d$-uniform clutter. This easily concludes \cite[Proposition 6]{DAli-Floystad-Nematbakhsh-01} where a description for multidegrees in the first linear strand of letterplace ideals are given, see also \cite[Proposition 5.8]{Nematbakhsh-02}.
\end{remark}

The simplicial complex $Y(I)$ has a simple structure. It is the simplicial complex that has the complements of $\supp(a_i)$ for $i=1,\ldots,d$ as its facets.
\[
Y(I) = < V(C_{|W}^c) \backslash \supp(a_i) ~|~ i=1,\ldots,d>.
\] 
It is easy to show that $Y(I)$ is the Stanley-Reisner complex of the Alexander dual of the ideal $\mfa$.
We show that it has the same homotopy type as a sphere.

\begin{lemma}
The complex $Y(I)$ is homotopy equivalent to a $(d-2)$-sphere.
\end{lemma}

\begin{proof}
For simplicity assume that $I$ is the edge ideal of a $d$-partite $d$-uniform clutter $C$ with a $d$-partition $V=V_1\cup \cdots\cup V_d$. Now $Y(I)$ is the simplicial complex generated by the facets $F_i=V\backslash V_i$ for $i=1,\ldots,d$.
For any $J\subseteq [d]$, $\cap_{i\in J} F_i \neq \emptyset$ if and only if $J=[d]$.
Therefore, the nerve complex of $Y(I)$ is the boundary of a $(d-1)$-simplex, which implies that $Y(I)$ is homotopy equivalent to a $(d-2)$-sphere.
\end{proof}

\begin{corollary}
The first linear strand of the edge ideals of $d$-partite $d$-uniform clutters are characteristic free. Furthermore, all of the nonzero multigraded Betti numbers on the first linear strand are equal to $1$.
\end{corollary}

\begin{remark}
Let $I$ be the edge ideal of a $d$-partite $d$-uniform clutter $C$.
If $I$ has a linear resolution then $R/I^A$ is a Cohen--Macaulay ring and by definition $\Ext^d_R(R/I^A,R(-\mathbf{1})) \cong I(C^c)^A/\mfa$ is its canonical module $\omega_{R/I^A}$. In this case the resolution in Theorem \ref{thm-main} is isomorphic to $\mcL_\bullet(\omega_{R/I^A})[d]$ and it gives the minimal free resolution of $I$.
\end{remark}

\begin{example} \label{exa-01}
Let $C$ be a $1$-partite $1$-uniform clutter and let $V_1 = \{x_1,\ldots,x_n\}$.
The ideal $I(C) = \mfm = <x_1,\ldots,x_n>$ is the homogeneous maximal ideal of $R = \kk[x_1,\ldots,x_n]$.
In this extreme case the complex in Theorem \ref{thm-main} is the Koszul complex.
\end{example}

\begin{example} \label{exa-02}
Let $P,Q$ be two finite posets.
Let $\Hom(P,Q)$ be the set of all order preserving maps.
For $\phi\in \Hom(P,Q)$, the graph of $\phi$,
\[
\Gamma\phi = \{(p,\phi(p)) ~|~ p\in P\}
\]
gives a monomial $m_{\Gamma\phi}$ in the polynomial ring $\kk[x_{P\times Q}]$. The ideal generated by all such monomials is called the {\it ideal of poset homomorphisms} of $P$ to $Q$ and is denoted by $L(P,Q)$, i.e $L(P,Q) = <m_{\Gamma\phi} | \phi \in \Hom(P,Q) >$.
The ideals of poset homomorphisms first appeared and studied in \cite{Floystad-Greve-Herzog-01} as a generalization of multichain ideals and their Alexander duals in \cite{Ene-Herzog-Mohammadi-01}.
Suppose $P=\{p_1,\ldots,p_d\}$. Let $C$ be the clutter with vertex set $P\times Q$ and edge set $\{\Gamma\phi~|~ \phi\in \Hom(P,Q)\}$. The clutter $C$ is $d$-partite and $d$-uniform with partition sets $V_i = \{(p_i,q)~|~q\in Q\}$. When $Q$ is a chain $[m]=\{1<\cdots<m\}$, it is shown in \cite{Ene-Herzog-Mohammadi-01}, that the ideal $L(P,[m])$ has linear quotients. Hence it has a linear resolution. 
In the same article, the authors construct the linear resolutions by iterrated mapping cones, see \cite{Herzog-Takayama-01}.
More generally, for any poset ideal $\mcJ \subseteq \Hom(P,[m])$, the ideal $L(\mcJ) = <m_{\Gamma\phi}~|~\phi\in \mcJ>$ has linear quotients as well. In \cite{DAli-Floystad-Nematbakhsh-02}, the linear resolution of these ideals are given using the Yanagawa construction. The canonical module is constructed differently in that article.
\end{example}

\begin{example} \label{exa-03}
The minimal free resolution of edge ideals of Ferrers graphs and hypergraphs are studied in \cite{Corso-Nagel-01,Nagel-Reiner-01}.
The Ferrers graphs are bipartite by definition, hence they are $2$-partite $2$-uniform clutters. A.~Corso and U.~Nagel in \cite{Corso-Nagel-01} show that Ferrers graphs are exactly the bipartite graphs for which their edge ideals have linear resolutions.
They explicitly construct the linear resolutions as well.
In \cite{Nagel-Reiner-01}, U.~Nagel and V.~Reiner construct cellular resolutions of edge ideals of $d$-uniform strongly stable and squarefree strongly stable clutters as specializations of resolutions of the corresponding $d$-uniform $d$-partite clutters.
We should mention that resolutions of strongly stable and squarefree strongly stable ideals in general are given in \cite{Eliahou-Kervaire-01,Aramova-Herzog-Hibi-01}.
\end{example}

\section{Applications in computation of Lyubeznik Numbers}
\label{sec-Lyubeznik}
Let $(R,\mfm)$ be a regular local ring of dimension $n$ containing a field $\kk$.
Let $M$ be an arbitrary $R$-module.
The $p$-th {\it Bass number} of $M$ with respect to the prime ideal $\mfp$ is the number
\[
\mu_p(\mfp,M) = \dim_{\kappa(\mfp)} \Ext^p_R(\kappa(\mfp),M_\mfp).
\]
Bass numbers are invariants of $M$ that describe the structure of the minimal injective resolution of $M$, see \cite{Bass-01}.
More precisely, if $I^\bullet$ is the minimal injective resolution of $M$, then the $p$-th term of $I^\bullet$ is
\[
I^p \cong \bigoplus_{\mfp\in \spec(R)} E_R(R/\mfp)^{\mu_p(\mfp,M)}.
\]
Let $I$ be an ideal of $R$.
The {\it $i$-th local cohomology} of $M$ with respect to $I$ is denoted by $H^i_I(M)$ and is defined to be
\[
H^i_I(M) = \underset{n}\varinjlim  \Ext^i_R(R/I^n,M).
\]
Finiteness of Bass numbers of local cohomology modules of $R$ in positive characteristic was shown by C.~Huneke and R.~Y.~Sharp \cite{Huneke-Sharp-01} and it was shown by G.~ Lyubeznik in zero characteristic \cite{Lyubeznik-01}.
Now let $I$ be an ideal of $R$ and let $A=R/I$.
The Bass numbers,
\[
\lambda_{p,i} (A) = \mu_p(\mfm,H^{n-i}_I(R))
\] 
are numerical invariants of $A$ and are called {\it Lyubeznik numbers}.
The Lyubeznik numbers were first introduced and studied in \cite{Lyubeznik-01}.
The Lyubeznik numbers enjoy the property that $\lambda_{p,i}(A) = 0$ for $p>i$ and $i>\dim A$.
Therefore, one can represent them in the form of an upper triangular matrix
\[
\Lambda(A) = \begin{pmatrix}
\lambda_{0,0} &\cdots & \lambda_{0,\dim A}\\
& \ddots & \vdots\\
&&\lambda_{\dim A,\dim A}
\end{pmatrix}
\]
called the {\it Lyubeznik Table}.
It is well-known that $\lambda_{\dim A,\dim A}$ is always nonzero.
We say that the Lyubeznik table is {\it trivial} if $\lambda_{\dim A,\dim A} =1$ and the rest of the entries of the table are zero.
The theory works just as well for a local graded ring $(R,\mathfrak{m})$ (e.g. a polynomial ring) and a homogeneous ideal $I$.
For a survey on Lyubeznik numbers and their generalizations see \cite{NunezBetancourt-Witt-Zhang-01}.
The relation between linear strands and Lyubeznik numbers has been
given in \cite{AlvarezMontaner-Vahidi-01, AlvarezMontaner-Yanagawa-01}.

Suppose $I$ is a squarefree monomial ideal in a polynomial ring $R=\kk[x_1,\ldots,x_n]$.
Let
\[F^{<r>}_\bullet(I): 
0 \to F^{<r>}_{n-r} \to \cdots \to F^{<r>}_1 \to F^{<r>}_0 \to 0
\]
be the $r$-th linear strand of $I$.
Following \cite{AlvarezMontaner-Vahidi-01}, by transposing the monomial matrices of $F^{<r>}_\bullet(I)$, we get a complex of $\kk$-vector spaces
\[
F^{<r>}_\bullet(I)^\ast: 0 \leftarrow K_0^{<r>} \leftarrow \cdots \leftarrow K^{<r>}_{n-r-1} \leftarrow K^{<r>}_{n-r} \leftarrow 0.
\]
In \cite{AlvarezMontaner-Vahidi-01}, J.~{\`A}lvarez Montaner and A.~Vahidi showed that
\[
\lambda_{p,n-r}(R/I^A) = \dim_\kk \H_p(\mcF_\bullet^{<r>}(I)^\ast).
\]

Now let $I$ be an ideal in a polynomial ring $R=\kk[x_1,\ldots,x_n]$ with an admissible regular sequence and let $d=d(I)$.
The Alexander dual of $I$ is a monomial ideal of height $d$. Therefore, $\dim (R/I^A)=n-d$ and the rightmost nonzero column of the Lyubeznik table is the $(n-d+1)$-th one.
It follows from \cite[Corollary 4.2]{AlvarezMontaner-Vahidi-01} and Theorem \ref{thm-main2} that the Lyubeznik numbers that lie on the last column of the Lyubeznik table of $I^A$ are Betti numbers of the relative simplicial complex $(X(I),Y(I))$ with coefficients in $\kk$.

For a squarefree monomial ideal $I \subseteq \kk[x_1,\ldots,x_n]$, the {\it Stanley-Reisner complex} $\Delta_I$ of $I$ is the simplicial complex on vertex set $\{x_1,\ldots,x_n\}$ consisting of the support of monomials not in $I$, i.e $\Delta_I = \{\supp(m) | m\notin I\}$.
The Hochster's formula states that the multigraded Betti numbers of $I$ can be computed from the reduced cohomology of subcomplexes of $\Delta_I$ with coefficients in $\kk$. More precisely, for $W\subseteq \{x_1,\ldots,x_n\}$
\[
\beta_{i,W}(I) = \dim_\kk \tH^{|W|-i-2}(\Delta_I|_W; \kk),
\]
where $\Delta_I|_W$ is the induced subcomplex of $\Delta_I$ on $W$.
We refer the reader to \cite[Chapter 8]{Herzog-Hibi-01} for more details.

\begin{theorem}
\label{thm-Lyu}
Let $I \subseteq R$ be a squarefree monomial ideal with an admissible regular sequence of length $d$.
We have
\[
\lambda_{p,n-d}(R/I^A) = \dim_\kk \tH^{n-p-1}((X(I),Y(I));\kk). 
\]
Furthermore, for $p<n-d$, 
\[
\lambda_{p,n-d}(R/I^A) = \dim_\kk \tH^{n-p-1}(X;\kk) = \beta_{p-1,n}(I(C^c)).
\]
\end{theorem}

\begin{proof}
Let $X=X(I)$ and $Y=Y(I)$.
The complex $F_\bullet^{<d>}(I)^\ast$ is isomorphic to the $\kk$-dual of the chain complex $\tC_\bullet(X,Y)$ with coefficients in $\kk$ (up to some shift).
More precisely, for the $i$-th term of $F_\bullet^{<d>}(I)^\ast$, we have
\[
F_i^{<d>}(I)^\ast \cong \Hom_\ZZ(\tC_{n-i-1}(X,Y),\kk)
\]
Therefore, the dimension of the cohomology spaces of this relative simplicial complex with coefficients in $\kk$ gives the Lyubeznik numbers on the last column of the Lyubeznik table.
\[
\lambda_{p,n-d}(R/I^A) = \dim_\kk \H_p(F_\bullet^{<d>}(I)^\ast) = \dim_\kk \tH^{n-p-1}((X,Y);\kk). 
\]
This completes the proof of the first part.
The exact sequence of chain complexes
\[
0 \to \Hom_\ZZ(\tC_\bullet(X,Y),\kk)  \to \Hom_\ZZ(\tC_\bullet(X),\kk) \to \Hom_\ZZ(\tC_\bullet(Y),\kk)  \to 0
\]
induces the long exact sequence on reduced cohomology
\[
\cdots \to \tH^i((X,Y);\kk) \to \tH^i(X;\kk) \to \tH^i(Y;\kk) \to \tH^{i+1}((X,Y);\kk) \to \cdots
\]
Since the only nonvanishing reduced cohomology of $Y$ is $\tH^{d-2}(Y;\kk)$, we have
an exact sequence
\begin{align*}
0\to \tH^{d-2}((X,Y);\kk)  \to \tH^{d-2}(X;\kk) \to & \tH^{d-2}(Y;\kk) \to \\
  &  \tH^{d-1}((X,Y);\kk) \to \tH^{d-1}(X;\kk) \to 0
\end{align*}
and we also have
\[
\tH^i((X,Y);\kk) = \tH^i(X;\kk) \text{ for all } i>d-1 \text{ and } i<d-2.
\]
Since $I$ is squarefree, the projective dimension of $I$ is less than or equal to $n-d$.
It follows from the Hochster formula that for $i<d-2$, $\tH^i(X;\kk) = 0$.
For $p<n-d$,
\begin{align*}
\lambda_{p,n-d}(R/I^A) &= \dim \tH^{n-p-1}((X,Y);\kk)\\
&= \dim \tH^{n-p-1}(X;\kk)\\
& = \beta_{{p-1},\mathbf{1}}(I(C^c)).
\end{align*}
The last equality follows from the Hochster's formula.
\end{proof}

There are several ways to compute Lyubeznik numbers of squarefree monomial ideals (see \cite[Section 8]{NunezBetancourt-Witt-Zhang-01}).

For a squarefree monomial ideal $I$ of an $n$-dimensional polynomial ring $R$, Yanagawa gave the following formula for computation of Lyubeznik numbers (see \cite[Corollary 3.10]{Yanagawa-03}).
\[
\lambda_{p,i}(R/I) = \dim_\kk [\Ext^{n-p}_R(\Ext^{n-i}_R(R/I,R),R)]_{\mathbf{0}}
\]
where $[M]_{\mathbf{0}}$ denotes the degree $\mathbf{0}$ component of a $\ZZ^n$-graded module $M$.

This is the simplest method but definitely not the fastest to implement in a computer algebra system with predefined $\Ext$ functions such as Macaulay2 \cite{Grayson-Stillman-01}.
Also, it is not difficult to implement in Macaulay2 the approach of \cite{AlvarezMontaner-Vahidi-01} based on the homology of linear strands of the resolution of $I$ (see \cite{AlvarezMontaner-FernandezRamos-01} for a detailed example).
The computer algebra system Macaulay2 is used to compute the Lyubeznik table in the examples below.


\begin{example}
Consider a $3$-partite $3$-uniform clutter with $3$-partition $\{a_1,a_2\}\cup \{b_1,b_2\}\cup\{c_1,c_2\}$ given by the following edge set
\[
\{\{a_1,b_1,c_1\},\{a_2,b_2,c_2\}\}.
\]
Let $C$ be its $4$-partite complement.
Let $R=\kk[a_1,a_2,b_1,b_2,c_1,c_2]$.
The edge ideal of $C$ is
\[
I(C) = <a_2b_1c_1,a_1b_2c_1,a_2b_2c_1,a_1b_1c_2,a_2b_1c_2,a_1b_2c_2>.
\]
The Lyubeznik table of $I(C)^A$ is
\[\Lambda(R/I(C)^A) = 
\begin{pmatrix}0&
      0&
      1&
      0\\
      0&
      0&
      0&
      0\\
      0&
      0&
      0&
      1\\
      0&
      0&
      0&
      1\\
      \end{pmatrix}
\]
Theorem \ref{thm-Lyu} implies that since $\lambda_{2,3}(R/I(C)^A) = 1$, the Betti number $\beta_{1,6}(I(C^c))$ equals $1$ as can be seen from the resolution  of $I(C^c)$
\[
\xymatrix{
I(C^c) \longleftarrow R(-3)^2 \longleftarrow S(-6)^1.
}
\]
\end{example}

\begin{example}
Let $C$ be a $4$-partite $4$-uniform clutter with $4$-partition
\[
V(C) = \{a_1,a_2\}\cup\{b_1,b_2\}\cup\{c_1,c_2\}\cup\{d_1,d_2\},
\]
and edge ideal
\[
I(C) = <a_1b_1c_1d_1,a_1b_2c_1d_1,a_1b_1c_2d_1,a_1b_2c_1d_2,a_1b_1c_2d_2,a_1b_2c_2d_2>.
\]
in the polynomial ring $R=\kk[V(C)]$.
The Lyubeznik table of $I(C)^A$ is
\[
\Lambda(R/I(C)^A) =
\begin{pmatrix}0&
      0&
      0&
      0&
      0\\
      0&
      0&
      0&
      1&
      0\\
      0&
      0&
      0&
      0&
      0\\
      0&
      0&
      0&
      0&
      1\\
      0&
      0&
      0&
      0&
      1\\
      \end{pmatrix}.
\]
The ideal $I(C^c)$ is an ideal of projective dimension $3$ with a minimal free resolution of the form
\[
I(C^c) \longleftarrow R(-4)^{10} \longleftarrow R(-5)^{14} \oplus R(-7)^1 \longleftarrow
R(-6)^6 \oplus R(-8)^1 \longleftarrow R(-7)^1.
\]
One can easily see that $\lambda_{3,4}(R/I(C)^A) = \beta_{2,8}(I(C^c)) = 1$.
\end{example}

\begin{example}
Consider the monomial ideal
\[
I = <a_1b_1,a_2b_1,a_2b_2,a_3b_2,a_2b_3,a_4b_3,a_1b_4,a_3b_4,a_4b_4>.
\]
The ideal $I$ is the edge ideal of a $2$-partite $2$-uniform clutter $C$ with $2$-partition $\{a_1,\ldots,a_4\}\cup \{b_1,\ldots,b_4\}$.
Let $R=\kk[a_1,\ldots,a_4,b_1,\ldots,b_4]$.
We have
\[
\Lambda(R/I^A) =
\begin{pmatrix}0&
       0&
       0&
       0&
       0&
       0&
       0\\
       0&
       0&
       0&
       0&
       0&
       0&
       0\\
       0&
       0&
       0&
       0&
       0&
       0&
       0\\
       0&
       0&
       0&
       0&
       0&
       2&
       0\\
       0&
       0&
       0&
       0&
       0&
       0&
       0\\
       0&
       0&
       0&
       0&
       0&
       0&
       2\\
       0&
       0&
       0&
       0&
       0&
       0&
       1\\
       \end{pmatrix}.
\]
The edge ideal of $C^c$ is the ideal
\[
I(C^c) =  <a_3b_1,a_4b_1,a_1b_2,a_4b_2,a_1b_3,a_3b_3,a_2b_4>
\]
with a minimal free resolution
\[
I(C^c) \longleftarrow R(-2)^7 \longleftarrow \begin{matrix} R(-3)^6\\ \bigoplus\\  R(-4)^9 \end{matrix} \longleftarrow \begin{matrix} R(-5)^{12} \\ \bigoplus \\ R(-6)^3 \end{matrix} \longleftarrow \begin{matrix} R(-6)^2 \\ \bigoplus \\ R(-7)^6 \end{matrix} \longleftarrow R(-8)^2.
\]

We can see that $\lambda_{5,6}(R/I^A) = \beta_{4,8}(I(C^c)) = 2$.
\end{example}

In \cite{AlvarezMontaner-Vahidi-01}, it is said that one may think of the Lyubeznik numbers of a squarefree monomial $I$ as a measure of the acyclicity of the $r$-linear strand of the Alexander dual $I^A$.
For sure their vanishing is a necessary condition for acyclicity of the linear strands but the following example shows that even for edge ideals of $d$-partite $d$-uniform clutters it is not enough.

\begin{example}
Let $I$ be the ideal
\begin{align*}
I=&  <a_2b_1c_1d_1,a_1b_2c_1d_1,a_2b_2c_1d_1,a_1b_1c_2d_1,a_2b_1c_2d_1,a_1b_2c_2d_1,a_2b_2c_2d_1,\\
  &  a_1b_1c_1d_2,a_2b_1c_1d_2,a_1b_2c_1d_2,
       a_2b_2c_1d_2,a_1b_1c_2d_2,a_2b_1c_2d_2,a_2b_2c_2d_2>
\end{align*}
The Betti diagram of $I$ is
\[\beta(I)=
\bordermatrix{
 & 0 & 1&2 & 3\cr
4& 14&24&12&1 \cr
5& \text{.}&\text{.}&1&1 \cr
}
\]
and the Alexander dual of $I$ is the ideal
\[
I^A=<a_1a_2,b_1b_2,c_1c_2,d_1d_2,a_2b_1c_1d_1,a_2b_2c_2d_2>
\]
An easy computation shows that Lyubeznik table of $I^A$ is trivial. 

The minimal free resolution of $I$ in multidegree $\mathbf{1} = (1,\ldots,1)$, i.e. the multidegree $m=a_1a_2b_1b_2c_1c_2d_1d_2$ has nontrivial homology only in homological degree $0$. But in multidegree $\frac{m}{a_2}$ it has nontrivial homology in degrees $0$ and $1$.

\end{example}

\section{Appendix}
\label{sec-appendix}
Here we give a characterization for edge ideals of $d$-partite $d$-uniform clutters with a linear resolution which is practically given in \cite{Favacchio-Guardo-Migliore-01} but is written in a different language.
In \cite{Favacchio-Guardo-Migliore-01}, the authors give a characterization for arithmetically Cohen-Macaulay sets of points in the multiprojective space $(\PP^1)^{\times d}=\PP^1\times\cdots\times \PP^1$ which is equivalent to characterization of edge ideals of $d$-partite $d$-uniform clutters with a linear resolution.
Unlike a finite set of points in a projective space which is always arithmetically Cohen-Macaulay, a finite set of points in a multiprojective space is not necessarily arithmetically Cohen-Macaulay.
There are many characterization for arithmetically Cohen-Macaulay sets of points in $\PP^1\times\PP^1$, 
see \cite[Chapter 4]{Guardo-VanTuyl-01}.
Recently, Favacchio et. al. gave a characterization of arithmetically Cohen-Macaulay sets of points in $(\PP^1)^{\times d}$ for $d\geq 2$ \cite{Favacchio-Guardo-Migliore-01}.
The characterization of arithmetically Cohen-Macaulay sets of points in the general case of $\PP^{d_1}\times\cdots\times \PP^{d_n}$ is still an open problem.

Let $\epsilon_1,\ldots, \epsilon_d$ be the unit vectors of $\ZZ^d$.
Let $R=\kk[x_0^1,x_1^1,x_0^2,x_1^2,\ldots,x_0^d,x_1^d]$ be a polynomial ring of dimension $2d$. The polynomial ring $R$ is $\ZZ^d$-graded by letting $\deg x_0^i = \deg x_1^i =\epsilon_i$, for $i=1,\ldots,d$.
Let 
\[P = ([a^1_0,a^1_1],[a^2_0,a^2_1],\ldots,[a^d_0,a^d_1])\]
be a point in $(\PP^1)^{\times d}$.
We will write $P$ as $[a^1_0,a^1_1]\times [a^2_0,a^2_1]\times\ldots\times[a^d_0,a^d_1]$.
For $i=1,\ldots,d$, let $A^i = a_1^i x_0^i - a_0^i x_1^i$ be the  linear form defining the point $[a_0^i,a_1^i]$ in $\PP^1$.
Note that $A^i$ defines a hyperplane in $\PP^1$.
The defining ideal of $P$ in $R$ is the multihomogeneous ideal $I_P = <A^1,\ldots,A^d>$.
We first describe how finite sets of points in $(\PP^1)^{\times d}$ correspond to $d$-partite $d$-uniform clutters. Then we recall the characterization of arithmetically Cohen-Macaulay sets of points in $(\PP^1)^{\times d}$ given in \cite{Favacchio-Guardo-Migliore-01}. In Proposition \ref{pro-linearchar}, we provide the characterization of $d$-partite $d$-uniform clutters with a linear resolution.



Let $X = \{P_1,\ldots,P_r\}$ be a finite set of points in $(\PP^1)^{\times d}$.
For a linear form $A$ we denote its corresponding hyperplane by $H_A$.
For $i$, $1\leq i\leq d$, let $r_i$ be the number of linear forms $A$ of degree $\epsilon_i$ such that $H_A \cap X \neq \emptyset$. We shall denote these linear forms by $A^i_1,\ldots,A^i_{r_i}$. 
For any point $P_j$ of $X$, $1\leq j\leq r$, we have  $I_{P_j} = <A^1_{j_1},\ldots,A^d_{j_d}>$
where $1\leq j_i \leq r_i$ for $i=1,\ldots,d$.
The defining ideal of $X$ is the multihomogeneous ideal $I_X= I_{P_1}\cap \cdots\cap I_{P_r}$.
The ideal $I_X$ also defines a union of linear varieties in $\PP^{2d-1}$.

Now let $C$ be a $d$-partite $d$-uniform clutter with vertex set 
\[
\{a^1_1,\ldots,a^1_{r_1}\} \cup \{a^2_1,\ldots, a^2_{r_d}\} \cup\cdots\cup \{a^d_1,\ldots
,a^d_{r_d}\}.
\]
Each vertex $a^i_j$ corresponds to a linear form $A^i_j$ defined above.
The edge set of $C$ is the union of sets $\{a^1_{j_1},\ldots,a^d_{j_d}\}$ where $a^1_{j_1},\ldots,a^d_{j_d}$ correspond to the ideal $<A^1_{j_1},\ldots,A^d_{j_d}>$ defining the component $P_j$ of $X$.
Let $S = \kk[a^i_j ~|~ i=1,\ldots,d,j=1,\ldots,r_i]$ be a polynomial ring.
Let $J$ be the Alexander dual of the edge ideal $I(C)$ of $C$ which is a monomial ideal in $S$. We consider this ideal as an ideal in $T=S[x_0^1,x_1^1,\ldots,x^d_0,x^d_1]$ and denote it by $\bar{J}$.
Let $L$ be the ideal generated by linear forms $a^i_j - A^i_j$ for $i=1,\ldots,d$ and $j=1,\ldots,r_i$.
Obviously, $R/I_X \cong T/(\bar{J}+L)$.
In \cite{Favacchio-Guardo-Migliore-01}, the authors imply that since $R/I_X$ and $T/\bar{J}$ both have height $d$, we can view the addition of each linear form in $L$ as a proper hyperplane section.

More precisely, it is easy to show that the linear forms in $L$ along with the forms $x_0^1,x_1^1,\ldots,x^d_0,x^d_1$ provide a basis for $T_1$.
Let $f:T\to T$ be the change of coordinates corresponding to this basis, which means that for $1\leq i\leq d$ and $1\leq j\leq r_i$, $f$ maps $a^i_j$ to $a^i_j - A^i_j$ and for $1\leq i\leq d$ and $1\leq j\leq 2$, it maps $x^i_j$ to itself.
We have $f^{-1}(L) = <a^i_j~|~ 1\leq i\leq d, 1\leq j\leq r_i>$ and 
\[
f^{-1}(\bar{J}) = \bigcap <a^1_{j_1} + A^1_{j_1}, \ldots, a^d_{j_d} + A^d_{j_d}>.
\]
where the intersection is over components of $X$.
Now it is easy to show that the sequence of variables 
\[
a^1_1,\ldots,a^1_{r_1},a^2_1,\ldots, a^2_{r_d},\ldots, a^d_{r_d}
\]
form a $T/f^{-1}(\bar{J})$-sequence.
This implies that the linear forms in $L$ form a $T/\bar{J}$-sequence.
Therefore, $R/I_X$ is Cohen-Macaulay if and only if $S/J$ is Cohen-Macaulay. It follows from a result of Eagon and Reiner \cite[Theorem 3]{Eagon-Reiner-01} that $X$ is arithmetically Cohen-Macaulay if and only if $I(C)$ has a linear resolution.

\begin{example}
Let $C$ be a clutter with the following edge set
\[
E(C) = \{
\{a_1,b_1,c_1\},\{a_1, b_1,c_2\},\{a_2,b_2,c_2\}
\}
\]
Now let $X$ be the union of $3$ points 
\[\{ [1,1]\times[1,1]\times[1,1],[1,1]\times[1,1]\times[2,1],[2,1]\times[2,1]\times[2,1]\}\]
in $\PP^1\times\PP^1\times \PP^1$.
We have
\begin{align*}
I_X = <x^1_0 - x^1_1, x^2_0-x^2_1, x^3_0 - x^3_1> & \cap
 <x^1_0 - x^1_1, x^2_0-x^2_1, x^3_0 - 2 x^3_1>\\
& \cap <x^1_0 - 2 x^1_1, x^2_0- 2 x^2_1, x^3_0 -2 x^3_1>
\end{align*}
In this example,
\[
L = <a_1 - x^1_0 - x^1_1, a_2 - x^1_0 - 2 x^1_1, b_1 - x^2_0 - x^2_1, \ldots,c_2 - x^3_0 - 2 x^3_1>.
\]
This forms a regular sequence on the ideal
\[
\bar{J} = <a_1,b_1,c_1>\cap <a_1,b_1,c_2> \cap <a_2,b_2,c_2>
\] 
It is easy to see that the Alexander dual of the monomial ideal $\bar{J}$ is precisely the edge ideal of the clutter $C$.

\end{example}

Let $P$ and $Q$ be two points in $(\PP^1)^{\times d}$.
Following \cite{Favacchio-Guardo-Migliore-01}, we denote by $Y_{P,Q}$ a height $d$ multihomogeneous complete intersection of least degree containing $P$ and $Q$, for which each minimal generator is a product of at most two hyperplanes of the same multidegree.
 
The main result of \cite{Favacchio-Guardo-Migliore-01} states that a finite set of points $X$ in $(\PP^1)^{\times d}$ is arithmetically Cohen-Macaulay if and only if it has the {\it $(\star_d)$-property} defined as follows: For any integer $d'$, $2\leq d'\leq d$, there do not exist two points $P$ and $Q$ in $(\PP^1)^{\times d}$ satisfying either of the following:
\begin{enumerate}
\item $P,Q\in X$ such that the ideal defining $Y_{P,Q}$ has exactly $d'$ minimal generators of degree 2 and $X\cap Y_{P,Q} = \{P,Q\}$;
\item $P,Q\notin X$ such that the ideal defining $Y_{P,Q}$ has exactly $d'$ minimal generators of degree 2 and $X\cap Y_{P,Q} = Y_{P,Q} \backslash \{P,Q\}$.
\end{enumerate}

Recall that the Fr\"oberg's theorem states that the edge ideal of a simple graph has a linear resolution if and only if its complement is chordal.
It easily follows from the Fr\"oberg's theorem that the edge ideal of a bipartite graph has a linear resolution if and only if it does not have any induced subgraphs consisting of two disjoint edges, see for example \cite[Section 6]{Nematbakhsh-02}.
In analogy to this result we have the following characterization for edge ideals of $d$-partite $d$-uniform clutters with linear resolutions.

\begin{proposition}
\label{pro-linearchar}
Let $C$ be a $d$-partite $d$-uniform clutter with a $d$-partition $V_1\cup\cdots\cup V_d$.
The edge ideal $I(C)$ has a linear resolution if and only if
for any integer $2\leq d' \leq d$ and any two edges $e$ and $e'$ of $C(V_1,\ldots,V_d)$,
neither the induced clutter on $e\cup e'$ nor its $d$-partite complement has any ranked projection of rank $d'$ consisting of two disjoint edges.
\end{proposition}

\begin{proof}
Let $P$ and $Q$ be the corresponding points in $(\PP^1)^{\times d}$ to $e$ and $e'$ respectively.
Note that the corresponding clutter to $Y_{P,Q}$ is the complete clutter $C(V_1\cap E,\ldots,V_d\cap E)$ where $E=e\cup e'$.
The first condition of $(\star_d)$-property states that there is a subset $J\subseteq [d]$ of cardinality $d'$ such that the projection to $W=\cup_{i\in J} V_i$ consists of two disjoint edges.
The second condition says that its complement has such a projection.
Therefore, $I(C)$ has a linear resolution if and only if none of the above happens.
This completes the proof. 
\end{proof}

In particular, if the edge ideal $I(C)$ has a linear resolution then the edge ideal of any ranked projection of $C$ has a linear resolution as well.
The following corollary is an easy but nontrivial consequence of Proposition \ref{pro-linearchar}.

\begin{corollary}
Let $C$ be a $d$-partite $d$-uniform clutter. The ideal $I(C)$ has a linear resolution if and only if $I(C^c)$ has a linear resolution.
\end{corollary}

Let $I \subseteq \kk[V]$ be a squarefree monomial ideal with minimal generators in a single degree $d$.
Let $\Delta$ be its Stanley-Reisner complex.
By Hochster formula, the ideal $I$ has a linear resolution if and only if for all $W\subseteq V$, $\H_i(\Delta_{|W},\kk) = 0$ when $i\neq d-2$. Note that for $i<d-2$, $\H_i(\Delta_{|W},k)$ vanishes.
Therefore, $I$ has a linear resolution if and only if for all $i>d-2$, $\H_i(\Delta_{|W},\kk) = 0$.
The content of Theorem \ref{pro-linearchar} implies that we can restrict our attention only to subsets $W$ of the form $W= \supp(\lcm (f,g))$ for $f,g\in G(I)$.

It is worth mentioning that if the edge ideal of a $d$-partite $d$-uniform clutter $C$ has a linear resolution then any ranked projection of $C$ or $C^c$ of rank greater or equal to $2$ does not have any induced subclutters consisting of two disjoint edges. The converse is not true.

\begin{example}
Let $C$ be a clutter with the following edge set
\[
E(C) =\{ \{a_1,b_1,c_2\}, \{a_2,b_2,c_2\},\{a_2,b_1,c_1\}\}.
\]
The clutter $C$ is a $3$-partite $3$-uniform clutter with $3$-partition 
\[
\{a_1,a_2\}\cup\{b_1,b_2\}\cup\{c_1,c_2\}.
\]
The edge ideal of $C$ does not have a linear resolution.
The Betti table of $I(C)$ is
\[\beta(I(C))=
\begin{pmatrix}
3&\text{.}&\text{.}\\
\text{.}&3&1\\
\end{pmatrix}
\]
Obviously, $C$ does not have any induced subclutter consisting of two disjoint edges.
For $W_1=\{a_1,a_2,b_1,b_2\}, W_2=\{a_1,a_2,c_1,c_2\}$ and $W_3=\{b_1,b_2,c_1,c_2\}$, the ranked projection clutters $C^{W_1},C^{W_2}$ and $C^{W_3}$ are Ferrers graphs. Hence it has no pairs of disjoint edges as induced subgraph. One can show that the same is also true for $C^c$.
\end{example}

\begin{remark}
A well-known classification of arithmetically Cohen-Macaulay sets of points in $\PP^1\times\PP^1$ says that a finite set $X\subseteq \PP^1\times \PP^1$ is arithmetically Cohen-Macaulay if and only if it has the inclusion property with respect to either of the projection maps, see \cite[Theorem 4.11]{Guardo-VanTuyl-01}.
Finite sets of points in $\PP^1\times \PP^1$ correspond to bipartite graphs.
In \cite{Corso-Nagel-01} it is shown that the edge ideal of a bipartite graph has a linear resolution if and only if it is a Ferrers graph.
By the discussions above, a finite set of points $X\subseteq \PP^1\times\PP^1$ is arithmetically Cohen-Macaulay if and only if the edge ideal of the corresponding graph has a linear resolution.
One can easily show that $X$ has the inclusion property if and only if the corresponding graph is a Ferrers graph.
Therefore, surprisingly both of the results in \cite{Corso-Nagel-01} and \cite{Guardo-VanTuyl-01} actually prove the same thing but in completely different languages.
\end{remark}

\bibliographystyle{plain}
\bibliography{MyBib}

{\bf Address:} School of Mathematics, Institute for Research in Fundamental Sciences (IPM), P. O. Box: 19395-5746, Tehran, Iran

{\bf Email:} nematbakhsh@ipm.ir

\end{document}